\documentclass[aps,superscriptaddress,11pt,twoside]{elsarticle}
\usepackage{amsmath,latexsym,amssymb,verbatim, enumerate, graphicx,color}

\usepackage{theorem}
\newtheorem{definition}{Definition}[section]
\newtheorem{proposition}[definition]{Proposition}

\newtheorem{theorem}[definition]{Theorem}

\newtheorem{conjecture}[definition]{Conjecture}

\def\squareforqed{\hbox{\rlap{$\sqcap$}$\sqcup$}}
\def\qed{\ifmmode\squareforqed\else{\unskip\nobreak\hfil
\penalty50\hskip1em\null\nobreak\hfil\squareforqed
\parfillskip=0pt\finalhyphendemerits=0\endgraf}\fi}
\def\endenv{\ifmmode\;\else{\unskip\nobreak\hfil
\penalty50\hskip1em\null\nobreak\hfil\;
\parfillskip=0pt\finalhyphendemerits=0\endgraf}\fi}
\newenvironment{proof}{\noindent \textbf{{Proof~} }}{\qed}

\newenvironment{example}{\noindent \textbf{{Example~}}}{\qed}

\mathchardef\ordinarycolon\mathcode`\:
\mathcode`\:=\string"8000
\def\vcentcolon{\mathrel{\mathop\ordinarycolon}}
\begingroup \catcode`\:=\active
  \lowercase{\endgroup
  \let :\vcentcolon
  }

\newcommand{\nc}{\newcommand}
\nc{\rnc}{\renewcommand} \nc{\beq}{\begin{equation}}
\nc{\eeq}{{\end{equation}}} \nc{\bea}{\begin{eqnarray}}
\nc{\eea}{\end{eqnarray}} \nc{\beqa}{\begin{eqnarray}}
\nc{\eeqa}{\end{eqnarray}} \nc{\lbar}[1]{\overline{#1}}
\nc{\bra}[1]{\langle#1|} \nc{\ket}[1]{|#1\rangle}
\nc{\ketbra}[2]{|#1\rangle\!\langle#2|}
\nc{\braket}[2]{\langle#1|#2\rangle} \nc{\proj}[1]{|
#1\rangle\!\langle #1 |} \nc{\avg}[1]{\langle#1\rangle}
\rnc{\max}{\operatorname{max}} \nc{\rank}{\operatorname{rank}}
\nc{\conv}{\operatorname{conv}}
\nc{\smfrac}[2]{\mbox{$\frac{#1}{#2}$}}
\nc{\Tr}{\operatorname{Tr}}
\nc{\ox}{\otimes}
\nc{\dg}{\dagger}
\nc{\dn}{\downarrow} \nc{\cA}{{\cal A}} \nc{\cB}{{\cal B}}
\nc{\cC}{{\cal C}} \nc{\cD}{{\cal D}} \nc{\cE}{{\cal E}}
\nc{\cF}{{\cal F}} \nc{\cG}{{\cal G}} \nc{\cH}{{\cal H}}
\nc{\cI}{{\cal I}} \nc{\cJ}{{\cal J}} \nc{\cK}{{\cal K}}
\nc{\cL}{{\cal L}} \nc{\cM}{{\cal M}} \nc{\cN}{{\cal N}}
\nc{\cO}{{\cal O}} \nc{\cP}{{\cal P}} \nc{\cR}{{\cal R}}
\nc{\cS}{{\cal S}} \nc{\cT}{{\cal T}} \nc{\cU}{{\cal U}}
\nc{\cX}{{\cal X}} \nc{\cW}{{\cal W}} \nc{\cZ}{{\cal Z}}
\nc{\csupp}{{\operatorname{csupp}}}
\nc{\qsupp}{{\operatorname{qsupp}}} \nc{\var}{\operatorname{var}}
\nc{\rar}{\rightarrow} \nc{\lrar}{\longrightarrow}
\nc{\poly}{\operatorname{poly}}
\nc{\polylog}{\operatorname{polylog}}
\nc{\Lip}{\operatorname{Lip}} \nc{\1}{\openone}

\def\>{\rangle}
\def\<{\langle}

\def\e{\epsilon}

\nc{\glneq}{{\raisebox{0.6ex}{$>$}  \hspace*{-1.8ex} \raisebox{-0.6ex}{$<$}}}
\nc{\gleq}{{\raisebox{0.6ex}{$\geq$}\hspace*{-1.8ex} \raisebox{-0.6ex}{$\leq$}}}

\nc{\RR}{{{\mathbb R}}}
\nc{\CC}{{{\mathbb C}}}
\nc{\FF}{{{\mathbb F}}}
\nc{\HH}{{{\mathbb H}}}
\nc{\NN}{{{\mathbb N}}}
\nc{\ZZ}{{{\mathbb Z}}}
\nc{\PP}{{{\mathbb P}}}
\nc{\QQ}{{{\mathbb Q}}}
\nc{\UU}{{{\mathbb U}}}
\nc{\WW}{{{\mathbb W}}}
\nc{\EE}{{{\mathbb E}}}
\rnc{\SS}{{{\mathbb S}}}
\nc{\id}{{\operatorname{id}}}

\nc{\vholder}[1]{\rule{0pt}{#1}}

\nc{\ob}[1]{#1}

\def\beq{\begin {equation}}
\def\eeq{\end {equation}}

\nc{\eq}[1]{Eq.~(\ref{eq:#1})} \nc{\eqs}[2]{Eqs.~(\ref{eq:#1}) and
(\ref{eq:#2})}

\nc{\eqn}[1]{Eq.~(\ref{eqn:#1})}
\nc{\eqns}[2]{Eqs.~(\ref{eqn:#1}) and (\ref{eqn:#2})}

\nc{\region}{\cS\cW}

\begin{document}

\title{Intersection patterns of linear subspaces \protect\\ with the hypercube}


\author[nm]{Nolmar Melo}
\ead{nolmar.melo@ufvjm.edu.br}
\address[nm]{Universidade Federal dos Vales do Jequitinhonha e Mucuri,
              Departamento de Ci\^{e}ncias Exatas, \\
              Jardim S\~{a}o Paulo,
              39803-371 Te\'{o}filo Otoni, MG, Brasil}
\author[aw1,aw2]{Andreas Winter}
\ead{andreas.winter@uab.cat}
\address[aw1]{ICREA---Instituci\'o Catalana de la Recerca i Estudis Avan\c{c}ats, \\
             Pg. Llu\'is Companys, 23, ES-08001 Barcelona, Spain}
\address[aw2]{Departament de F\'isica: Grup d'Informaci\'o Qu\`antica, \\
             Universitat Aut\`onoma de Barcelona, ES-08193 Bellaterra (Barcelona), Spain}     

\renewcommand*{\today}{6 December 2017, revised 7 November 2018}        
\date{6 December 2017, revised 7 November 2018}

\begin{abstract}
Following a combinatorial observation made by one of us recently in
relation to a problem in quantum information 
[Nakata \emph{et al.}, \emph{Phys. Rev. X} \textbf{7}:021006 (2017)],
we study what are the possible intersection cardinalities
of a $k$-dimensional subspace with the hypercube in $n$-dimensional
Euclidean space. We also propose two natural variants of the 
problem by restricting the type of subspace allowed.

We find that whereas every natural number eventually occurs as the 
intersection cardinality for \emph{some} $k$ and $n$, on the other
hand for each fixed $k$, the possible intersections sizes are governed by 
severe restrictions.
To wit, while the largest intersection size is evidently $2^k$, there
is always a large gap to the second largest intersection size, which
we find to be $\frac34 2^k$ for $k \geq 2$ (and $2^{k-1}$ in the
restricted version). 
We also present several constructions, and propose a number of
open questions and conjectures for future investigation.
\end{abstract}

\maketitle

\parskip .75ex


\section{Introduction} 
\label{sec:intro}
The interplay between algebra and combinatorics has proved fruitful
in may different contexts and in different ways 
\cite{BabaiFrankl,BannaiIto,new-persp}. 
The particular flavour of problems we are looking at here is obtained
by importing combinatorial structures as $0$-$1$-vectors into a real
vector space and looking at those satisfying linear constraints; using 
these can result in elegant proofs of combinatorial theorems \cite{Winter:Bondy}.
Ahlswede \emph{et al.} were among the first to systematically
consider extremal problems under dimension constraints 
\cite{AAK-0,AAK-1,AAK-5}. Our work is more in line
with a result by Alon and F\"uredi on covering the hypercube by 
hyperplanes \cite{AlonFueredi}.

In the present paper, we want to further the deep connection between these 
two fields by considering one of the most basic questions one can ask about 
the hypercube $\HH^n := \{0,1\}^n \subset \RR^n$, namely what
can be its intersection with a $k$-dimensional affine linear
subspace $S$ of $\RR^n$? 

By way of notation, we write $e_j$ for the $j$-th standard basis vector of, 
$j=1,\ldots,n$. For a generic vector $x\in\RR^n$, denote $(x)_k$ its $k$-th
component, so that $x = \sum_{k=1}^n (x)_k e_k$.

Obviously, the largest cardinality of
an intersection is $2^k$, but are all other numbers smaller than that
possible? Here, we want to study the sets $H(n,k)$ of possible non-zero
intersection cardinalities of a $k$-dimensional $S$ with $\HH^n$:
\begin{equation}
  H(n,k) := \bigl\{ t > 0: \exists S < \RR^n,\ \dim S = k,\ t = |\HH^n\cap S| \bigr\}.
\end{equation}
Except for the case of empty intersection, any affine subspace can be
translated and reflected, using the symmetries of the hypercube, to a linear subspace
(i.e.~containing the origin) that has the same number of points in
common with $\HH^n$ as the original subspace. We shall henceforth assume,
whenever it is convenient, that $S$ is a linear subspace.
In that case, furthermore, we may assume, by possibly permuting the 
coordinates of $\RR^n = \RR^k \oplus \RR^{n-k}$, 
that $S$ projected onto the first $k$ coordinates spans $\RR^k$.
(Here, we identify $\RR^k = \operatorname{span}\{e_1,\ldots,\e_k\}$ with the subspace of 
$\RR^n$ of vectors that have the last $n-k$ coordinates equal to $0$; likewise,
$\RR^{n-k} = \operatorname{span}\{e_{k+1},\ldots,\e_n\}$ is identified with the subspace of 
$\RR^n$ of vectors that have the first $k$ coordinates equal to $0$.)
Thus, $S$ can be parametrised as 
\[
  S = \bigl\{ v \oplus Lv : v \in \RR^k \bigr\},
  \text{ with a linear map } L:\RR^k\longrightarrow \RR^m,
\]
where $m=n-k$. We can write explicitly, $L=\pi_m \circ \pi_k^{-1}$, where
$\pi_k:S\longrightarrow\RR^k$ and $\pi_m:S\longrightarrow\RR^m$ are the
projection of $S$ onto the first $k$ and last $m$ coordinates, respectively.
This gives rise to an equivalent characterisation of $H(n,k)$ as
\begin{equation}
	H(n,k) = \bigl\{ t : \exists L:\RR^k\rightarrow \RR^m,\ t=|\HH^k \cap L^{-1}\HH^m| \bigr\}.
\end{equation}
This characterisation has the advantage that additional properties of the
linear map $L$ can be imposed. For instance, in \cite{t-design} the 
case of an isometry $L$ (and hence implicitly $m \geq k$) was studied,
motivating the definitions
\begin{equation}
  \widehat{H}(n,k) = \bigl\{ t : \exists L:\RR^k\rightarrow \RR^m \text{ isometry},\ 
                                        t=|\HH^k \cap L^{-1}\HH^m|=|\HH^m \cap L\HH^k| \bigr\},
\end{equation}
and the potentially more flexible
\begin{equation}
  \widetilde{H}(n,k) = \bigl\{ t : \exists L:\RR^k\rightarrow \RR^m \text{ contraction},\ 
  t=|\HH^k \cap L^{-1}\HH^m| \bigr\}.
\end{equation}

Clearly, when $k=n$, $H(n,k)=\widetilde{H}(n,k)=\{2^k\}$, so we may assume 
from now on that $n > k$, in which case both $H(n,k)$ and $\widetilde{H}(n,k)$
contain all powers of $2$ from $1$ to $2^k$. In \cite{t-design} it was shown
that the second largest element of $\widehat{H}(2k,k)$ is $2^{k-1}$, a result
we shall reproduce and generalise below.
Apart from the cases of fixed $n$ and $k$, we are very much interested in the
case of unbounded $n$ for given $k$, giving rise to the variants
\[
  H(\infty,k)             := \bigcup_{n\geq k} H(n,k),\quad
  \widehat{H}(\infty,k)   := \bigcup_{n\geq 2k} \widehat{H}(n,k),\quad
  \widetilde{H}(\infty,k) := \bigcup_{n\geq k} \widetilde{H}(n,k).   
\]

The structure of the rest of the paper is as follows: 
In section \ref{sec:constructions}, we review some constructions showing 
certain numbers to be realisable as intersections.
In section \ref{sec:Hnk} we derive bounds on the second largest
number in $H(n,k)$, and in section \ref{sec:tilde-Hnk} we find
the second largest number in the isometry and contraction versions
$\widehat{H}(n,k)$ and $\widetilde{H}(n,k)$. After that we conclude
highlighting several open questions and conjectures.

\section{Some constructions and other observations} 
\label{sec:constructions}

Evidently, the powers of \(2\) are always possible intersections,
\[
  \{1,2,4,\ldots,2^k\} \subset \widehat{H}(\infty,k) 
                       \subset \widetilde{H}(\infty,k)
                       \subset H(\infty,k).
\]
A simple construction for $2^j \in \widehat{H}(2k,k)$ is in terms 
of the following isometry $L:\RR^k \longrightarrow \RR^k$:
\[
  L e_i := \begin{cases}
              e_i & \text{ for } i \leq j, \\
             -e_i & \text{ for } j < i \leq k.
           \end{cases}
\]
To see $2^j \in \widetilde{H}(k+1,k)$, consider the contraction
$L:\RR^k \longrightarrow \RR$ defined by
\[
  L e_i := \begin{cases}
                            0 & \text{ for } i \leq j, \\
             -\varepsilon e_i & \text{ for } j < i \leq k,
           \end{cases}
\]
with $0 < \varepsilon < \frac 1k$ any sufficiently small positive number.

However, powers of $2$ are not the whole story, by a long way.
To start, $3 \in H(3,2)$ as can be seen from inspecting the
linear map $L:\RR^2 \longrightarrow \RR$ defined by
$L e_1 = L e_2 = 1$. The following proposition generalises this
observation.

\begin{proposition}
  \label{prop:k+1:2k-1}
  For all $k\geq 1$: $k+1 \in H(k+1,k)$, $2k-1 \in H(k+2,k)$
  and $2^{k-1}+1 \in H(2k-1,k)$.
\end{proposition}
\begin{proof}
1. Consider the linear map $L:\RR^k \longrightarrow \RR$ defined by
$L e_i = 1$ for all $i=1,\ldots,k$. The only $v\in\HH^k$
such that $L v \in \HH^1$ are evidently $0$ and the $e_i$; any other 
vector $v\in \HH^k$ is a sum of at least two $e_i$'s, so $L v \geq 2$.

2. For the second number, consider $L:\RR^k \longrightarrow \RR^2$ defined by
$L e_i = e_1 + e_2$ for $i=1,\ldots,k-1$, and $L e_k = -e_2$. 
By inspection, the only $v\in\HH^k$ such that $L v \in \HH^2$ are evidently 
$0$, the $e_i$ and $e_i+e_k$, for $i=1,\ldots,k-1$.

3. For the third, consider $L:\RR^k \longrightarrow \RR^{k-1}$ defined by
$L e_i = e_i$ for $i=1,\ldots,k-1$, and $L e_k = e_1+e_2+\ldots+e_{k-1}$. 
Again, it is easy to see that the $v \in \HH^k$ with $L v \in\HH^{k-1}$
are precisely $v=e_k$ and any $v$ supported on the first $k-1$ 
coordinates, $v=\sum_{i=1}^{k-1} v_i e_i$ with $v_i\in\{0,1\}$ arbitrary.
\end{proof}

\medskip
These constructions show that every natural number eventually appears
in some $H(n,k) \subset H(\infty,k)$, for sufficiently large $k$.
Hence the ``right'' question is which numbers appear in $H(\infty,k)$ 
for a given $k$.

\begin{proposition}
  \label{prop:increase-n-k}
  For all $n\leq n'$ and $k\leq k'$, such that $n'-n \geq k'-k$,
  \[
    H(n,k) \subset H(n',k'), \ 
    \widetilde{H}(n,k) \subset \widetilde{H}(n',k'), \ 
    \widehat{H}(n,k) \subset \widehat{H}(n',k'), 
  \]
  where in the last relation we implicitly require $n\geq 2k$ and $n' \geq 2k'$.
\end{proposition}
\begin{proof}
Let \( t \in H(n,k)\setminus \{0\}  \), so there is a linear map 
\(L:\RR^k\rightarrow \RR^m\) where \(t= |\HH^k\cap L^{-1}\HH^m |\). 
If \( L':\RR^{k+1}\rightarrow \RR^m \) is a linear map with \(L'|_{\RR^k}=L \) 
and \(L'(e_{k+1})=(b,0,\ldots,0)\) with \(b>\sum_{i=1}^k |(L(e_i))_1|+1 \) 
we have that \(L'(e_{k+1}+v)\notin \HH^m\) for all \(v\in \HH^k \), 
so \( |\HH^k\cap L'^{-1}\HH^m |=t \) that imply \(t\in H(n,k+1)\).

Now let \( t\in H(n+1,k)\) and consider a map \(L\) that realises \(t\).
If we define \(L':\RR^k\rightarrow \RR^{m+1}\) 
with \(L'(v)=(L(v),0)\), we have that \(t= |\HH^{k}\cap L'^{-1}\HH^{m+1} |\). 
So we can conclude that  $H(n,k) \subset H(n',k')$ for all $n \leq n'$, $k \leq k'$.
\end{proof}

\medskip
As with any combinatorial problem, when the structure is not 
evident from the start, we begin by experimenting with small numbers.
By inspection, and aided by the above constructions, we see that
\[
  H(\infty,1) = \{1,2\} \text{ and } H(\infty,2) = \{1,2,3,4\}.
\]
Furthermore,
\begin{equation}
  \label{eq:H(3)}
  H(\infty,3) \supset \{1,2,3,4,5,6,8\},
\end{equation}
where the cardinalities $3$ and $5$ come from Proposition~\ref{prop:k+1:2k-1}.
Also $6$ could be constructed directly, but we use the occasion
to point out a general principle, the direct sum construction
(Proposition~\ref{prop:directsum} below),
which gives us $6\in H(5,3)$, as $3\in H(3,2)$ times $2\in H(2,1)$.
The same principles give us
\begin{equation}
  \label{eq:H(4)}
  H(\infty,4) \supset \{1,2,3,4,5,6,7,8,9,10,12,16\}.
\end{equation}

\begin{proposition}
  \label{prop:directsum}
  If $t\in H(n,k)$ and $t'\in H(n',k')$, then $t\cdot t' \in H(n+n',k+k')$.
\end{proposition}
\begin{proof}
As \(t\in H(n,k)\), there is \(L:\RR^k\rightarrow\RR^m\) with 
\( t=|\HH^k \cap L^{-1}\HH^m|\); the same happens with  
\( t'=|\HH^{k'} \cap L'^{-1}\HH^{m'}|\). 
Let \(\Lambda:\RR^k\oplus \RR^{k'}= \RR^{k+k'}\rightarrow \RR^{m+m'} = \RR^m\oplus \RR^{m'}\) 
be defined by \(\Lambda(v\oplus v')=L(v)\oplus L'(v')\). 
Observe that \(\Lambda(\HH^k\oplus 0) \cap \Lambda(0 \oplus \HH^{k'})=\{0\} \), 
so \(v\oplus v'\in \Lambda^{-1}\HH^{m+m'} \) if and only if 
\(v\in L^{-1}\HH^m\) and \(v'\in L'^{-1}\HH^{m'}\), which gives \(t\cdot t'\) possibilities.
\end{proof}

\medskip
We observe the gaps appearing in the lists for $k=3$ and $k=4$:
while the lower ranges are filling up in accordance with the
constructions given in Proposition~\ref{prop:k+1:2k-1} and the
other observations we made, close to the maximum we could not
find any subspaces realising cardinalities just under $2^k$.
This is no coincidence or lack of imagination, as we show in the
next section. Indeed, Theorem~\ref{thm:Hnk-gap} below shows
$7\not\in H(\infty,3)$ and $13,\,14,\,15 \not\in H(\infty,4)$; 
this leaves as the only number unresolved $11$ in the list of $H(\infty,4)$.

More generally, for every $k$, we have $2^k-2^{k-2} = 3\cdot 2^{k-2} \in H(k+1,k)$ 
by applying the direct sum construction to $3\in H(3,2)$ and $2^{k-2} \in H(k-2,k-2)$;
and it turns out that this is the second largest number in $H(\infty,k)$
(Theorem~\ref{thm:Hnk-gap} in the next section).

\medskip
We now show some more constructions of a general nature.

\begin{proposition}
  \label{prop:binom:knapsack}
  For all $n \geq \ell$, $k \geq \ell-1$, we have ${\ell \choose r} \in H(n,k)$.
  More generally, consider a knapsack problem with weights $p_i \geq 0$ ($i=1,\ldots,\ell$)
  and a total weight $q \geq 0$,
  \[
    \left|\left\{ v\in\HH^\ell : \sum_{i=1}^\ell p_i v_i = q \right\}\right| \in H(n,k).
  \]
  In other words, the number of solutions of a knapsack problem, i.e.~the
  number of ways of writing $q$ as a sum of a subset of the numbers $\{p_i\}$,
  is in $H(n,k)$.
\end{proposition}
\begin{proof}
By Proposition~\ref{prop:increase-n-k}, we only have to prove the claims
for $n=\ell$ and $k=\ell-1$. The first one is a special case of the second
by letting all $p_i=1$ and $q=r$.

To prove the second claim, notice that the equation $\sum_{i=1}^\ell p_i v_i = q$
defines an affine hyperplane (subspace of dimension $\ell-1$) in $\RR^\ell$,
thus by definition, its cardinality of intersection with $\HH^\ell$ is in
$H(\ell,\ell-1)$.
\end{proof}

\begin{proposition}
  \label{prop:binom+1}
  We have \( {\ell \choose r}+1 \in H(n,k) \) when \(\ell \leq k < n\).
\end{proposition}
\begin{proof}
Let the linear map \(L:\RR^k\rightarrow \RR\) be given by	
\[
  Le_i=\begin{cases}
         \frac 1r & \text{for } i\leq \ell,\\
         2	& \text{for } i>\ell.
       \end{cases}
\]
Then, \(v\in \HH^k\) is mapped to $\HH$ if and only if 
\(v=0\) or \(v=e_{j_1}+\cdots +e_{j_r}\) with \(j_i\neq j_{i'}\) 
when \(i\neq i'\) and \(j_i\in [\ell]=\{1,2\ldots ,\ell\}\). So we have that 
\( {\ell \choose r}+1 \in H(n,k) \) when \(\ell \leq k\).
\end{proof}

\begin{proposition}
For all \(k\geq 1\), and \(r<t<k\), it holds that  
\(2^t+2^r\in H(\infty,k)\).
If \(t\leq k-2\), we have \(2^t+2^r+1\in H(\infty,k)\).
\end{proposition}
\begin{proof}
Let \(L:\RR^k \rightarrow \RR^k \) be a linear map given by
\[
  L e_i := \begin{cases}
                e_i                 & \text{ for } 1 \leq i \leq t, \\
                e_1+ \ldots +e_\ell & \text{ for } i = t+1,         \\
               -e_i                 & \text{ for } t+2 \leq i \leq k,
            \end{cases}
\]        
where \(\ell=t-r\). It is easy to see that \(2^t+2^r=|\HH^k \cap L^{-1}\HH^k | \) ,
which means \( 2^t+2^r\in H(2k,k) \subset H(\infty,k) \).

When \(t<k-1\), we modify the previous map \(L\) by letting
\(L e_{t+2} = e_1+\cdots +e_{t+2}\). 
Again, it is easy to check that
\( |\HH^k \cap L^{-1}\HH^k | = 2^t+2^r+1\).
\end{proof}     

\medskip
We can generalise this construction as follows.

\begin{proposition}
  \label{prop:powers-of-2}
  Let \(k-j \geq t_j>\cdots>t_1>t_0\geq 0\). Then,
  \( \sum_{i=0}^j 2^{t_i} \in H(\infty,k) \).
\end{proposition}
\begin{proof}
We construct the linear map \(L:\RR^k \rightarrow \RR^k \) as follows:
\[
  L e_i := \begin{cases}
              e_i & \text{ for } i \leq t_j ,\\
             -e_i & \text{ for } t_j < i \leq k-j, \\
              \sum_{\nu=t_\delta+1}^{t_j} e_\nu & \text{ for } i = k-\delta,
            \end{cases}
\]
where the last case ranges over $\delta=0,\ldots,j-1$.

It is checked straightforwardly that $Lv \in \HH^k$, for 
$v=\sum_i v_i e_i \in \HH^k$, if and only if all $v_i=0$ ($t_j < i \leq k-j$),
and at most one of $v_{k-\delta}=1$ ($\delta=0,\ldots,j-1$); if all
are $0$, we set $\delta=j$, and then the $v_i$ with $i>t_j$
must be $0$, whereas the first $t_j$ are free in $\{0,1\}$.
This shows that \( |\HH^k \cap L^{-1}\HH^k | = \sum_{i=0}^j 2^{t_i}\).
\end{proof}

\medskip
We close this section by looking at what we can say about $k=5$:
From the above constructions we get directly
\[
  H(\infty,5) \supset \{1,2,3,4,5,6,7,8,9,10,12,14,15,16,17,18,20,24,32\}.
\]
However, also $11$ and $13 \in H(\infty,5)$, using the following 
joint generalisation of constructions one and three in 
Proposition~\ref{prop:k+1:2k-1} 
(alternatively by Proposition~\ref{prop:powers-of-2} above).

\begin{proposition}
  \label{prop:t+1}
  Let $t\in H(n,k)$ be such that there is a linear map
  $L:\RR^k \longrightarrow \RR^m$ with $t = \bigl| \HH^k \cap L^{-1}\HH^m \bigr|$,
  and the additional property that for every $v\in\HH^k\setminus \{0\}$
  there exists a coordinate $j\in[m]$ such that $(L v)_j > 0$.
  
  Then, $t+1 \in H(n+1,k+1)$, and there exists a linear map
  $\Lambda:\RR^{k+1} \longrightarrow \RR^m$ with 
  $t+1 = \bigl| \HH^k \cap \Lambda^{-1}\HH^m \bigr|$,
  which has the same positivity property as $L$.
\end{proposition}
\begin{proof}
As before, we consider \(\{e_1,\cdots, e_{k+1} \}\) as a linear basis  
of $\RR^{k+1}$ and identify \(\RR^k\) with the subspace spanned by \(\{e_1,\cdots, e_k\}\).
Then, $\Lambda$ is defined by letting
\[
  \Lambda e_i := \begin{cases}
                   L e_i              & \text{ for } i=1,\ldots,k, \\
                   e_1+e_2+\ldots+e_m & \text{ for } i=k+1.
                 \end{cases}
\]
This has the desired properties: Indeed, the positivity of
coordinates is inherited from $L$. Furthermore, $v\in \HH^{k+1}$
can be written as $v = v' + v_{k+1}e_{k+1}$ with $v'\in\HH^k$, so
$\Lambda v = L v' + v_{k+1}(e_1+\ldots+e_m)$, which is in $\HH^m$
iff $v_{k+1}=0$ and $L v'\in\HH^m$, or $v_{k+1}=1$ and $v'=0$.
\end{proof}

\section{Bounds on the largest numbers in $H(n,k)$} \label{sec:Hnk}

As we have observed, the largest element of $H(\infty,k)$ is trivially 
$2^k$; what is the gap to the second largest element, denoted $h_2(k)$?
It is easy to see that $h_2(1)=1$ and $h_2(2)=3$, and we have observed 
above that $h_2(k) \geq 2^k-2^{k-2}$ in general. As it turns out, this 
is tight.

\begin{theorem}
  \label{thm:Hnk-gap}
  For any $k \geq 2$, the second largest number in $H(\infty,k)$ equals
  $h_2(k) = 2^k-2^{k-2}$. It is contained in $H(n,k)$ for all $n > k$.
\end{theorem}

\begin{proof}
We distinguish a number of different cases. It will be crucial to
consider the images of the basis vectors, $f_i := L e_i$, and for
each of them the \emph{support} $F_i = \{ j\in[k] : (f_i)_j \neq 0 \}$ 
of indices where $f_i$ has non-zero coefficients.

\begin{enumerate}
\item If for all $i\in[k]$, $f_i := L e_i\in \HH^m$ and for all $i\neq j$,
  $L(e_i+e_j)\in \HH^m$, it means that the $f_i$ are the characteristic
  functions of the \emph{disjoint} sets $F_i$, $F_i\cap F_j = \emptyset$.
  Thus, for all $x \in \HH^k$, we have by linearity $L x = \sum_i x_i f_i \in \HH^m$,
  and so the intersection is $2^k$.

\item If still for all $i\in[k]$, $f_i = L e_i \in \HH^m$, but there exists
  a pair $i\neq j$, such that $L(e_i+e_j) \not\in \HH^m$, then 
  the assumption means that $F_i\cap F_j\neq \emptyset$, and hence $L(e_i+e_j)$
  has a coordinate equal to $2$.
  But since $L x$ for $x \in\HH^k$ is component-wise non-negative,
  we then have that $L(x+e_i+e_j) \not\in \HH^m$ for all
  $x \in \HH^{[k]\setminus\{i,j\}} \subset \HH^k$ (we think of these
  as the elements of $\HH^k$ with $i$th and $j$th coordinate $0$).
  Thus, the intersection is at most $2^k-2^{k-2}$.

\item It remains to consider the case that for some $i$, $L e_i \not\in \HH^m$,
w.l.o.g.~$i=k$. If the intersection size $\bigl| L^{-1} \HH^m \cap \HH^k \bigr|$ is
  at most $2^{k-1}$, we are done; if it is 
  strictly larger than $2^{k-1}$, we can use the pigeon hole principle
  to deduce that there exists a $v\in\HH^{k-1}$,
  such that $e_k+v\in\HH^k$ and $L v,\, L(e_k+v)\in\HH^m$. Hence,
  $L e_k = L(e_k+v) - L v \in \{0,1,-1\}^m =: \EE^m$, where we have
  introduced notation for the \emph{extended hypercube}.
  By assumption $L e_k \not\in\HH^m$, so this vector must have one
  coordinate $-1$, w.l.o.g.~the $m$th: $(L e_k)_m = -1$.
  Also, necessarily $(L v)_m = 1$, and because $L v = \sum_{i=1}^{k-1} v_i L e_i$,
  there must exist an $i < k$ such that $(L e_i)_m > 0$, w.l.o.g.~$i=k-1$.
  Now, consider any $u\in \HH^{k-2}$: Either $(L u)_m \not\in \{1,2\}$,
  but then $\bigl( L(u+e_k) \bigr)_m \not\in \{0,1\}$ and so $L(u+e_k) \not\in \HH^m$;
  or $(L u)_m \in \{1,2\}$, but then it follows $\bigl( L(u+e_{k-1}) \bigr)_m > 1$ and so
  $L(u+e_{k-1}) \not\in \HH^m$. 
  In any case, we find $2^{k-2}$ vectors $u'\in\HH^k$ such that $L u'\not\in\HH^m$,
  hence $\bigl| L^{-1} \HH^m \cap \HH^k \bigr| \leq 2^k-2^{k-2}$.
\end{enumerate}

This concludes the proof, because we found that all of the intersection
numbers $\bigl| L^{-1} \HH^m \cap \HH^k \bigr|$ are either equal to $2^k$
or are $\leq 2^k-2^{k-2}$. 
\end{proof}

\bigskip
So, what about the third largest number? Note that Theorem~\ref{thm:Hnk-gap}
says that $\sup_k 2^{-k} h_2(k) = \frac34$, and the maximum is attained 
already for $k=2$. So it may make sense to look at it in terms of
$\sup_k 2^{-k} h_3(k)$, etc. By Proposition~\ref{prop:k+1:2k-1} applied to
$k=3$, we get $5\in H(5,3)$, i.e.~a ratio of $\frac58$ is obtained. What about
$\frac{11}{16}$, or even larger ones? 
Note that in dimension $k$, Proposition~\ref{prop:k+1:2k-1} gives us 
$2^{k-1}+1\in H(\infty,k)$, so a ratio of $\frac12 + 2^{-k}$ is realised. 
Based on the evidence of cases we looked at, the following seems
reasonable.

\begin{conjecture}
\label{conj:large}
For all $k \geq 1$,
\[
  H(\infty,k) \cap \{n\in\NN : n > 2^{k-1}\} = \{2^{k-1}+2^{i} : 0 \leq i < k\}.
\]
To prove this, we would however need a far-reaching extension of the
method used to characterise the largest two numbers.
\end{conjecture}

We close this section with some observations on the small numbers
in $H(\infty,k)$. In the examples for $k=1,2,3,4,5$ that we looked at
above, we observe that all integers up to $2^{k-1}+2$ occur as 
intersections. We believe this to be a general pattern and propose
the following conjecture.

\begin{conjecture}
\label{conj:small}
For all $k \geq 1$, $H(\infty,k)$ contains the whole integer interval 
$\bigl[2^{k-1}+2\bigr] = \bigl\{1,2,\ldots,2^{k-1}+1,2^{k-1}+2\bigr\}$.

To prove this by induction, we would need a construction
that takes us from $t\in H(\infty,k)$ to $2t-1\in H(\infty,k+1)$,
because we have already $2t\in H(\infty,k+1)$ by the direct sum construction.
Note that we do not actually have a universal construction to show
$2^{k-1}-1\in H(\infty,k)$; numerical tests however seem to support this.
\end{conjecture}

\section{Bounds on the largest numbers in $\widehat{H}(n,k)$ and $\widetilde{H}(n,k)$} 
\label{sec:tilde-Hnk}
The isometric case, and also the more general contractive case,
are much more constrained, as has been observed before.
Evidently, both $\widehat{H}(2k,k)$ and $\widetilde{H}(k+1,k)$
contain the powers of $2$, $\{1,2,4,\ldots,2^{k-1},2^k\}$. At
least at the upper end that's all there is to it:

\begin{theorem}[{Orthogonal case in Nakata \emph{et al.}~\cite{t-design}}]
  \label{thm:tilde-Hnk-gap}
  For any $k$, the second largest number in $\widetilde{H}(\infty,k)$,
  and in particular the second largest number in $\widehat{H}(\infty,k)$,
  equals $\tilde{h}_2(k) = 2^{k-1}$.
  It is contained in $\widehat{H}(n,k)$ for all $n\geq 2k$ and in
  $\widetilde{H}(n,k)$ for all $n > k$.
\end{theorem}

\begin{proof}
It is clear that $2^{k-1}$ is in $\widehat{H}(2k,k)$, and in $\widetilde{H}(k+1,k)$.
For the opposite, nontrivial, claim, we have to show that for
a contraction $L:\RR^k \longrightarrow \RR^m$ with
$\bigl| L^{-1} \HH^m \cap \HH^k \bigr| > 2^{k-1}$, the intersection
must already be of the maximum size $2^k$.

By the pigeon hole principle, for every $i\in[k]$ there exists a 
$v_i\in\HH^{[k]\setminus i}$ such that $v_i,\, v_i+e_i \in \HH^k$ and 
$L v_i,\, L(v_i+e_i) \in \HH^m$. Thus, we get, by linearity of
the map, $L e_i = L(v_i+e_i) - L v_i \in \{0,1,-1\}^m$. Since 
$L$ is a contraction, it follows that either $L e_i = 0$ or
$L e_i = \mu_i e_{\lambda(i)}$ with $\mu_i=\pm 1$. An important
observation is that for $i\neq j$, necessarily $\lambda(i)\neq\lambda(j)$.
[For if it were otherwise, $\lambda(i)=\lambda(j)=\lambda$, 
consider $v=\mu_i e_i+\mu_j e_j$, which is mapped to $L v = 2e_\lambda$;
however, $v$ has length $\sqrt{2}$, whereas its image $L v$ has
length $2$, contradicting the assumption that $L$ is a contraction.]
We can include the case $L e_i=0$ by allowing $\mu_i=0$, and
introducing $\lambda(i)$ differing, w.l.o.g., from all other $\lambda(j)$.

Now we can express the image of a generic $x\in\HH^k$ as
\[
  L x = \sum_{i=1}^k x_i L e_i
      = \sum_{i=1}^k x_i \mu_i e_{\lambda(i)},
\]
which is in $\HH^m$ if and only if $x_i \mu_i \neq -1$ for all $i$,
meaning that $\bigl| L^{-1} \HH^m \cap \HH^k \bigr|$ is a power of $2$,
namely $2^{k-\nu}$ with $\nu = |\{ i : \mu_i=-1 \}|$. By our assumption,
$\nu = 0$, and thus the entire hypercube $\HH^k$ is mapped to $\HH^m$ by $L$.
\end{proof}

\medskip
%
%
What about the subsequent gaps, and the general structure of the 
sets $\widetilde{H}(\infty,k)$ and $\widehat{H}(\infty,k)$?
It turns out that it's not all powers of $2$.
To wit, it is of course the case that 
$\widehat{H}(\infty,1) = \widetilde{H}(\infty,1) = \{1,2\}$
and $\widehat{H}(\infty,2) = \widetilde{H}(\infty,2) = \{1,2,4\}$,
in accordance with Theorem~\ref{thm:tilde-Hnk-gap}. However, for larger
$k$, we have intersection numbers that are not powers of $2$.

\begin{example} 
For $k=3$, it holds
$\widetilde{H}(\infty,3) = \widehat{H}(\infty,3) = \widehat{H}(6,3) = \{1,2,3,4,8\}$.
Indeed, the possibility of all powers of two is evident (they are
evidently realised by isometries), 
and the gap between $4$ and $8$ is by Theorem~\ref{thm:tilde-Hnk-gap}, 
leaving only the cardinality $3$ to be realised. This is accomplished by the 
orthogonal map $L:\RR^3\longrightarrow\RR^3$ defined via
\begin{align*}
  L e_1 &= \frac13 \left(  e_1 + 2e_2 + 2e_3 \right), \\
  L e_2 &= \frac13 \left( 2e_1 +  e_2 - 2e_3 \right), \\
  L e_3 &= \frac13 \left( 2e_1 - 2e_2 +  e_3 \right),
\end{align*}
which has the property that $L 0 = 0$, $L(e_1+e_2)=e_1+e_2$,
$L(e_1+e_3)=e_1+e_3$, but no other hypercube point is mapped to
the hypercube as can be seen by inspection.

As a consequence, the third largest number $\hat{h}_3(k)$ in 
$\widehat{H}(\infty,k) \subset \widetilde{H}(\infty,k)$, is lower
bounded $\hat{h}_3(k) \geq 3\cdot 2^{k-3}$, and in fact it is found in
$\widehat{H}(2k,k)$, and also in $\widetilde{H}(k+3,k)$.
\end{example}

\medskip
In particular, the construction implies for $k=4$, that
$\widetilde{H}(\infty,4) \supset \widehat{H}(\infty,4) 
                         \supset \widehat{H}(8,4) \supset \{1,2,3,4,6,8,16\}$.
Theorem~\ref{thm:tilde-Hnk-gap} rules out the range from $9$ to $15$, leaving
however the questions after $5$ and $7$. The former is provided by the 
following construction, albeit in dimension larger than $8$.

\begin{proposition}
  \label{prop:k+1-isometry}
  For every $k\geq 3$, $n\geq 3k-2$, we have 
  $k+1 \in \widehat{H}(n,k) \subset \widehat{H}(\infty,k) \subset \widetilde{H}(\infty,k)$.
  In particular, $5\in \widehat{H}(10,4)$.
\end{proposition}

\begin{proof}
Denoting the all-$1$ vector in $\RR^k$ by $\vec{1}^k$, we introduce
$v_i = \vec{1}^k-e_i$ for $i=1,\ldots,k$. Note that these vectors all have
weight $k-1$ and pairwise inner product $k-2$. This allows us to define
an isometry $L:\RR^k \longrightarrow \RR^{2k-2}$ by letting
\[
  L v_i := e_i \oplus \vec{1}^{k-2} \quad (i=1,\ldots,k).
\]
By adding all of these equations we obtain
$(k-1)L \vec{1}^k = \vec{1}^k \oplus k\,\vec{1}^{k-2}$, and hence
\[\begin{split}
  L e_i &= L\left( \vec{1}^k-v_i \right) \\
        &= \frac{1}{k-1}\left( \vec{1}^k \oplus k\,\vec{1}^{k-2} \right) 
                                 - \left( e_i \oplus \vec{1}^{k-2} \right) \\
        &= \left( \frac{1}{k-1}\vec{1}^k - e_i \right) \oplus \frac{1}{k-1}\vec{1}^{k-2}.
\end{split}\]
This means, looking at the second terms in the direct sum,
that only those hypercube points (which are sums of $e_i$'s) can be
mapped to the hypercube that have weight $0$ or $k-1$, leaving precisely
the origin plus the $k$ vectors $v_i$.
\end{proof}

\section{Discussion}
\label{sec:discussion}
We have initiated the study of the possible intersections of a hypercube
with a linear subspace, or equivalently of the number of points of the
$k$-hypercube mapped to an $m$-hypercube by a linear map. 
While the largest such number is clearly $2^k$, we showed that the 
second-largest number is $\frac34$ of that in the general case
(and $\frac12$ in the case of isometries and contractions). It seems
that also the third-largest intersection, and perhaps more of the
``large'' intersection cardinalities are bounded away from $2^k$ by
a constant fraction gap.

On the other end, regarding ``small'' intersection cardinality,
we have empirically observed that for $k\leq 5$, all integers from
$1$ to $2^{k-1}+2$ occur, and conjecture that this is the case in
general. To prove this conjecture remains one of the biggest
open problems of our study.

Some other concrete questions that we would like to recommend to the
attention of the reader include the following:

\begin{enumerate}
\item Given \(t>0\), what is the smallest \(k\) such that \(t\in H(\infty,k)\)?
\item Given \(t\in H(\infty,k)\), what is the smallest \(n\) such that \(t\in H(n,k)\)?
\end{enumerate}

\vspace{.5cm}
\subsection*{Acknowledgments}
This work was initiated during NM's sabbatical with the
Grup d'Infor-maci\'o Qu\`antica (GIQ) at UAB, in the autumn
semester of the 2016/17 term.
The authors have nobody to thank but each other for the nice
mess they've gotten themselves into.

This research was supported by the EC (STREP FP7-ICT-2013-C-323970 ``RAQUEL''), 
the ERC (AdG ERC-2010-AdG-267386 ``IRQUAT''), 
the Spanish MINECO (grant nos. FIS2013-40627-P and FIS2016-86681-P) 
with the support of FEDER funds, as well as by the Generalitat de 
Catalunya CIRIT, project 2014-SGR-966.

\bigskip
\emph{Note added.} After posting of our manuscript on arXiv, Carla Groenland
and Tom Johnston (arXiv[math.CO]:1810.02729) managed to make progress with our
conjectures. In fact, they proved a modified version of our 
Conjecture \ref{conj:large} on the ``large'' elements $t>2^{k-1}$ of $H(\infty,k)$,
namely
\[
  H(\infty,k) \cap \{2^{k-1}+1,\ldots,2^k\} = \{2^{k-1}+2^{i}:i=0,\ldots,k-1\} \cup \left\{\frac{35}{64}2^k\right\}.
\]
They also disproved our Conjecture \ref{conj:small} on the ``small'' 
elements $t\leq 2^{k-1}$ of $H(\infty,k)$, amounting to the claim that $H(\infty,k)$
would contain all integers up to $2^{k-1}$, by exhibiting a constant
fraction of missing numbers. It seems that $H(\infty,k)$ has a much more intricate 
structure than anticipated.

\vspace{.5cm}
\centering{$\bullet$}
\vspace{.5cm}

\bibliographystyle{unsrt}

\end{document}